\documentclass[10pt,reqno]{amsart}

\usepackage{amsthm, mathrsfs, amsmath, amstext, amsxtra, amsfonts, dsfont, amssymb}
\usepackage{xcolor}
\usepackage{lmodern}
\usepackage[colorlinks, linkcolor=red, citecolor=blue, urlcolor=blue, pagebackref, hypertexnames=false]{hyperref}

\newcommand{\R}{\mathbb R}
\newcommand{\C}{\mathbb C}

\newcommand{\eps}{\epsilon}

\newcommand{\re}[1]{\mbox{Re} #1} 
\newcommand{\im}[1]{\mbox{Im} #1} 

\newcommand{\scal}[1]{\left\langle #1 \right\rangle} 

\newcommand{\defendproof}{\hfill $\Box$} 

\newtheorem{theorem}{Theorem}[section]
\newtheorem{lemma}[theorem]{Lemma} 
\newtheorem{proposition}[theorem]{Proposition}
 
\theoremstyle{definition}
\newtheorem{definition}[theorem]{Definition}
\newtheorem{remark}[theorem]{Remark}

\title[Stability Instability NLS inverse-square]{On stability and instability of standing waves for the nonlinear Schr\"odinger equation with inverse-square potential} 

\author{Abdelwahab Bensouilah}
\address{Laboratoire Paul Painlev\'e (U.M.R. CNRS 8524), U.F.R. de Math\'ematiques, Universit\'e Lille 1, 59655 Villeneuve d'Ascq Cedex, France}
\email{ai.bensouilah@math.univ-lille1.fr}

\author{Van Duong Dinh}
 \address{Institut de Mathematiques de Toulouse UMR5219, Universit\'e de Toulouse CNRS, 31062 Toulouse Cedex 9, France}
\email{dinhvan.duong@math.univ-toulouse.fr}

\author{Shihui Zhu}
\address{College of Mathematics and Software Science, Sichuan Normal University, Chengdu 610066, China; School of Mathematical Sciences,
University of Electronic Science and Technology of China, Chengdu, Sichuan 611731, China}
\email{shihuizhumath@163.com; shihuizhumath@uestc.edu.cn}


\begin{document}
	
	\begin{abstract}
		We consider the focusing nonlinear Schr\"odinger equation with inverse square potential
		\[
		i\partial_t u + \Delta u + c|x|^{-2} u = - |u|^\alpha u, \quad u(0) = u_0 \in H^1, \quad (t,x) \in \R^+ \times \R^d,
		\]
		where $d \geq 3$, $c\ne 0$, $c<\lambda(d)=\left(\frac{d-2}{2}\right)^2$ and $0<\alpha\leq \frac{4}{d}$. Using the profile decomposition obtained recently by the first author \cite{Bensouilah}, we show that in the $L^2$-subcritical case, i.e. $0<\alpha<\frac{4}{d}$, the sets of ground state standing waves are orbitally stable. In the $L^2$-critical case, i.e. $\alpha=\frac{4}{d}$, we show that ground state standing waves are strongly unstable by blow-up. 
	\end{abstract}	
	
    \maketitle
    
	\section{Introduction}
	\setcounter{equation}{0}
	
	Consider the focusing nonlinear Schr\"odinger equation with inverse-square potential
	\begin{align}
		\left\{
		\begin{array}{rcl}
			i\partial_t u + \Delta u + c|x|^{-2} u &=& - |u|^\alpha u, \quad (t,x)\in \R^+ \times \R^d, \\
			u(0)&=& u_0 \in H^1,
		\end{array} 
		\right. \label{focusing NLS inverse square}
	\end{align}
	where $d\geq 3$, $u: \R^+ \times \R^d \rightarrow \C$, $u_0:\R^d \rightarrow \C$, $c\ne 0$ satisfies $c<\lambda(d):=\left(\frac{d-2}{2}\right)^2$ and $0<\alpha\leq \frac{4}{d}$. 
    
     The Schr\"odinger equation \eqref{focusing NLS inverse square} appears in a variety of physical settings, such as quantum field equations or black hole solutions of the Einstein's equations \cite{Case, CEFC, KSWW}. The mathematical interest in the nonlinear Schr\"odinger equation with inverse-square potential comes from the fact that the potential is homogeneous of degree $-2$ and thus scales exactly the same as the Laplacian. 
     
	Let $P^0_c$ denote the natural action of $-\Delta - c|x|^{-2}$ on $C^\infty_0(\R^d\backslash \{0\})$. When $c \leq \lambda(d)$, the operator $P^0_c$ is a positive semi-definite symmetric operator. Indeed, we have the following identity
	\[
	\scal{\varphi, P^0_c \varphi} = \int |\nabla \varphi (x)|^2 - c|x|^{-2} |\varphi(x)|^2 dx  = \int \left|\nabla \varphi(x) + \rho x|x|^{-2} \varphi(x) \right|^2 dx \geq 0, 
	\]
	for all $\varphi \in C^\infty_0 (\R^d \backslash \{0\})$, where
	\[
	\rho:= \frac{d-2}{2} - \sqrt{\left(\frac{d-2}{2}\right)^2 -c}. 
	\]
	Denote $P_c$ the self-adjoint extension of $P^0_c$. It is known (see \cite{KSWW}) that in the range $\lambda(d)-1<c<\lambda(d)$, the extension is not unique. In this case, we do make a choice among possible extensions such as Friedrichs extension. Note also that the constant $\lambda(d)$ is the sharp constant appearing in Hardy's inequality
	\begin{align}
	\lambda(d) \int |x|^{-2} |u(x)|^2 dx \leq \int |\nabla u(x)|^2 dx, \quad \forall u \in H^1. \label{hardy inequality} 
	\end{align}
	Throughout this paper, we denote the Hardy functional
	\begin{align}
	\|u\|^2_{\dot{H}^1_c}:= \|\sqrt{P_c} u\|^2_{L^2}=\int |\nabla u(x)|^2 - c|x|^{-2} |u(x)|^2 dx, \label{hardy norm}
	\end{align}
	and define the homogeneous Sobolev space $\dot{H}^1_c$ as the completion of $C^\infty_0(\R^d \backslash \{0\})$ under the norm $\|\cdot\|_{\dot{H}^1_c}$. It follows from $(\ref{hardy inequality})$ that for $c<\lambda(d)$, 
	\begin{align}
	\|u\|_{\dot{H}^1_c} \sim \|u\|_{\dot{H}^1}. \label{equivalent gradient norm}
	\end{align}
	This is an assertion of the isomorphism between the homogeneous space $\dot{H}^1_c$ defined in terms of $P_c$ and the usual homogeneous space $\dot{H}^1$. We refer the interested reader to \cite{KMVZZ-sobolev} for the sharp range of parameters for which such an equivalence holds.

	The local well-posedness for $(\ref{focusing NLS inverse square})$ was established in \cite{OSY-energy}.  More precisely, we have the following result.
	\begin{theorem}[Local well-posedness \cite{OSY-energy}] \label{theorem lwp}
		Let $d\geq 3$ and $c\ne 0$ be such that $c<\lambda(d)$. Then for any $u_0 \in H^1$, there exists $T\in (0,+\infty]$ and a maximal solution $u \in C([0,T), H^1)$ of $(\ref{focusing NLS inverse square})$. The maximal time of existence satisfies either $T=+\infty$ or $T<+\infty$ and 
		\[
		\lim_{t\uparrow T} \|\nabla u(t)\|_{L^2} =\infty.
		\]
		Moreover, the solution enjoys the conservation of mass and energy, i.e. 
		\begin{align*}
		M(u(t)) &= \int |u(t,x)|^2 dx = M(u_0), \\
		E(u(t)) &= \frac{1}{2} \int |\nabla u(t,x)|^2 dx -\frac{c}{2} \int |x|^{-2} |u(t,x)|^2 dx -\frac{1}{\alpha+2} \int |u(t,x)|^{\alpha+2} dx, 	\end{align*}
		for any $t\in [0,T)$. Finally, if $0<\alpha<\frac{4}{d}$, then $T=+\infty$, i.e. the solution exists globally in time.
	\end{theorem}
	We refer the reader to \cite[Theorem 5.1]{OSY-energy} for the proof of this result. Note that the existence of solutions is based on a refined energy method and the uniqueness follows from Strichartz estimates. Note also that Strichartz estimates for the linear NLS with inverse-square potential were first established in \cite{BPSTZ} except the endpoint case $(2,2d/(d-2))$. Recently, Bouclet-Mizutani \cite{BM} proved Strichartz estimates with the full set of admissible pairs for the linear NLS with critical potentials including the inverse-square potential. The local well-posedness for $(\ref{focusing NLS inverse square})$ can be proved using Strichartz estimates and the equivalence between Sobolev spaces defined by $P_c$ and the usual ones via the Kato method. However, due to the appearance of inverse-square potential, the local well-posedness proved by Strichartz estimates requires a restriction on the validity of $c$ and $d$ (see e.g. \cite{ZZ, KMVZ-focusing, KMVZZ-energy, LMM}). 
	
	The main purpose of this paper is to study the stability and instability of standing waves for \eqref{focusing NLS inverse square}. In fact, the stability of standing waves for the nonlinear Schr\"odinger equations is widely pursued by physicists and mathematicians (see \cite{F}, for a review). For the classical nonlinear Schr\"odinger equation, Cazenave and Lions \cite{CL} were the first to prove the orbital stability of standing waves via the concentration-compactness principle. Then, a lot of results on the orbital stability were obtained. For the nonlinear Schr\"odinger equation with a harmonic potential, Zhang \cite{Z} succeed in obtaining the orbital stability by the weighted compactness lemma. Recently, the stability phenomenon was proved for the fractional nonlinear Schr\"{o}dinger equation by establishing the profile decomposition for bounded sequences in $H^s$(se  \cite{FZ,PS,Zhu}). 
    
  The first part of this paper concerns the stability of standing waves in the $L^2$-subcritical case $0<\alpha <\frac{4}{d}$.
Before stating our stability result, let us introduce some notations. For $M>0$, we consider the following variational problems
	\begin{itemize}
		\item for $0<c<\lambda(d)$,
		\begin{align}
		d_M:= \inf \left\{ E(v) \ : \ v \in H^1, \|v\|^2_{L^2}=M  \right\}; \label{variational problem 1}
		\end{align}
		\item for $c<0$,
		\begin{align}
		d_{M,\text{rad}}:= \inf \left\{ E(v) \ : \ v \in H^1_{\text{rad}}, \|v\|^2_{L^2}=M  \right\}, \label{variational problem 2}
		\end{align}
	\end{itemize}
    where $H^1_{\text{rad}}$ is the space of radial $H^1$-functions. Note that in the case $c<0$, we are only interested in radial data. This is related to the fact that the sharp Gagliardo-Nirenberg inequality for non-radial data (see Section $\ref{section preliminaries}$) is never attained when $c<0$. We will see later (Proposition $\ref{proposition variational problems}$) that the above variational problems are well-defined. Moreover, the above infimums are attained. Let us denote
	\begin{itemize}
		\item for $0<c<\lambda(d)$, 
		\[
		S_M:= \left\{ v \in H^1 \ : \ v \text{ is a minimizer of } (\ref{variational problem 1}) \right\};
		\]
		\item for $c<0$,
		\[
		S_{M,\text{rad}}:= \left\{ v \in H^1_{\text{rad}} \ : \ v \text{ is a minimizer of } (\ref{variational problem 2}) \right\}.
		\]
	\end{itemize}
	By the Euler-Lagrange theorem (see Appendix), we see that if $v \in S_M$, then there exists $\omega >0$ such that
	\begin{align}
	-\Delta v - c|x|^{-2} v + \omega v = |v|^\alpha v. \label{elliptic equation}
	\end{align}
	Note also that if $v$ is a solution to $(\ref{elliptic equation})$, then $u(t,x):= e^{i\omega t} v(x)$ is a solution to $(\ref{focusing NLS inverse square})$. One usually calls $e^{i\omega t} v$ the orbit of $v$. Moreover, if $v \in S_M$, i.e. $v$ is a minimizer of $(\ref{variational problem 1})$, then $e^{i\omega t} v$ is also a minimizer of $(\ref{variational problem 1})$ or $e^{i\omega t} v \in S_M$. A similar remark goes for $v \in S_{M,\text{rad}}$. 
	
	We next define the following notion of orbital stability which is similar to the one in \cite{TZ}. 
	\begin{definition} \label{definition orbital stability}
		The set $S_M$ is said to be {\bf orbitally stable} if, for any $\eps>0$, there exists $\delta>0$ such that for any initial data $u_0$ satisfying
		\[
		\inf_{v \in S_M} \|u_0 -v\|_{H^1} <\delta,
		\]
		the corresponding solution $u$ to $(\ref{focusing NLS inverse square})$ satisfies
		\[
		\inf_{v \in S_M} \|u(t)-v\|_{H^1}<\eps,
		\]
		for all $t\geq 0$. A similar definition applies for $S_{M,\text{rad}}$.
	\end{definition}
	Our first result is the following orbital stability of standing waves for the $L^2$-subcritical $(\ref{focusing NLS inverse square})$. 
	\begin{theorem}[Orbital stability] \label{theorem stability}
		Let $d\geq 3$, $0<\alpha<\frac{4}{d}$ and $M>0$. 
		\begin{enumerate}
			\item If $0<c<\lambda(d)$, then $S_M$ is orbitally stable.
			\item If $c<0$, then $S_{M,\emph{rad}}$ is orbitally stable.
		\end{enumerate}
	\end{theorem}

   Let us mention that the stability of standing waves in the case $0<c<\lambda(d)$ was studied in \cite{TZ}. However, they only considered radial standing waves in this case. Here, we remove the radially symmetric assumption and prove the stability for non-radial standing waves in the case $0<c<\lambda(d)$. Moreover, our approach is based on the profile decomposition which is of particular interest. We also study the stability of radial standing waves in the case $c<0$, which to our knowledge is new. 
   
   The proof of our stability result is based on the profile decomposition related to $(\ref{focusing NLS inverse square})$. Note that this type of profile decomposition was recently established by the first author in \cite{Bensouilah}. The main difficulty is the lack of space translation invariance due to the inverse-square potential. A careful analysis is thus needed to overcome the difficulty. We refer the reader to Section $\ref{section stability}$ for more details.
   
   The second part of this paper is devoted to the strong instability result in the $L^2$-critical case $\alpha=\frac{4}{d}$. There are two main difficulties in studying this problem. The first difficulty is the lack of regularity of solutions to the elliptic equation
\begin{align}
-\Delta Q - c|x|^{-2} Q + Q = |Q|^{\frac{4}{d}} Q. \label{elliptic equation introduction}
\end{align}
More precisely, we do not know whether $Q \in L^2(|x|^2dx)$ for any solution $Q$ of $(\ref{elliptic equation introduction})$. This is a strong contrast with the classical NLS ($c=0$) where solutions to $(\ref{elliptic equation introduction})$ are known to have an exponential decay at infinity. Another difficulty is that the uniqueness (up to symmetries) of positive radial solutions to $(\ref{elliptic equation introduction})$ is not yet known. To overcome these difficulties, we need to define properly the notion of ground states. To do this, we follow the idea of Csobo-Genoud \cite{CG} and define the set of ground states $\mathcal{G}$ and the set of radial ground states $\mathcal{G}_{\text{rad}}$ (see Section $\ref{section instability}$ for more details). Using this notion of ground states, we are able to show that any ground state $Q$ satisfies $Q \in L^2(|x|^2 dx)$, similarly any radial ground state $Q_{\text{rad}}$ satisfies $Q_{\text{rad}} \in L^2(|x|^2 dx)$. Thanks to this fact, the standard virial identity yields the following instability in the $L^2$-critical case.
\begin{theorem} \label{theorem instability I}
		\begin{enumerate}
			\item Let $d\geq 3, 0<c<\lambda(d)$ and $Q \in \mathcal{G}$. Then the standing wave $e^{it} Q(x)$ is unstable in the following sense: there exists $(u_{0,n})_{n\geq 1} \subset H^1$ such that
			\[
			u_{0,n} \rightarrow Q \text{ strongly in } H^1,
			\]
			as $n\rightarrow \infty$ and the corresponding solution $u_n$ to the $L^2$-critical $(\ref{focusing NLS inverse square})$ with initial data $u_{0,n}$ blows up in finite time for any $n\geq 1$.
			\item Let $d\geq 3, c<0$ and $Q_{\emph{rad}} \in \mathcal{G}_{\emph{rad}}$.
			Then the radial standing wave $e^{i t} Q_{\emph{rad}}(x)$ is unstable in the following sense: there exists $(u_{0,n})_{n\geq 1} \subset H^1$ such that
			\[
			u_{0,n} \rightarrow Q_{\emph{rad}} \text{ strongly in } H^1,
			\]
			as $n\rightarrow \infty$ and the corresponding solution $u_n$ to the $L^2$-critical $(\ref{focusing NLS inverse square})$ with initial data $u_{0,n}$ blows up in finite time for any $n\geq 1$.
		\end{enumerate}
	\end{theorem}
If we are interested in radial $H^1$-solutions to $(\ref{elliptic equation introduction})$, then we can show another version of instability of standing waves. The interest of this instability is that it allows radial $H^1$-solutions of $(\ref{elliptic equation introduction})$ whose $L^2$-norms are greater than the $L^2$-norms of ground states. We refer the reader to Section $\ref{section instability}$ for more details.

   The paper is organized as follows. In Section $\ref{section preliminaries}$, we recall sharp Gagliardo-Nirenberg inequalities and the profile decomposition related to $(\ref{focusing NLS inverse square})$. In Section $\ref{section stability}$, we give the proof of the stability result stated in Theorem $\ref{theorem stability}$. Finally, we study the strong instability of standing waves in Section $\ref{section instability}$.

		
	\section{Preliminaries} 
    \label{section preliminaries}
	\setcounter{equation}{0}
	\subsection{Sharp Gagliardo-Nirenberg inequalities}
	In this section, we recall the sharp Gagliardo-Nirenberg associated to $(\ref{focusing NLS inverse square})$, namely
	\begin{align}	
	\|u\|^{\alpha +2}_{L^{\alpha+2}} \leq C_{\text{GN}}(c) \|u\|^{\frac{4-(d-2)\alpha}{2}}_{L^2} \|u\|^{\frac{d\alpha}{2}}_{\dot{H}^1_c}, \label{gagliardo nirenberg inequality}
	\end{align}	
	for all $u \in H^1$. The sharp constant $C_{\text{GN}}(c)$ is defined by	
	\[	
	C_{\text{GN}}(c):= \sup \left\{ J^\alpha_c(u) \ : \ u \in H^1 \backslash \{0\} \right\},	
	\]	
	where $J^\alpha_c(u)$ is the Weinstein functional	
	\begin{align}
	J^\alpha_c(u):= \frac{\|u\|^{\alpha+2}_{L^{\alpha+2}} }{  \|u\|^{\frac{4-(d-2)\alpha}{2}}_{L^2} \|u\|^{\frac{d\alpha}{2}}_{\dot{H}^1_c} }.	\label{weinstein functional}
	\end{align}
	We also recall the sharp radial Gagliardo-Nirenberg inequality, namely	
	\begin{align}	
	\|u\|^{\alpha+2}_{L^{\alpha+2}} \leq C_{\text{GN}}(c,\text{rad}) \|u\|^{\frac{4-(d-2)\alpha}{2}}_{L^2} \|u\|^{\frac{d\alpha}{2}}_{\dot{H}^1_c}, \label{radial gagliardo-nirenberg inequality}	
	\end{align}	
	for all $u \in H^1_{\text{rad}}$. The sharp constant $C_{\text{GN}}(c,\text{rad})$ is defined by	
	\[	
	C_{\text{GN}}(c,\text{rad}) : = \sup \left\{ J^\alpha_c(u) \ : \ u \in H^1_{\text{rad}} \backslash \{0\} \right\}.
	\]
	In the case $c=0$, it is well known (see \cite{Weinstein}) that the sharp constant $C_{\text{GN}}(0)$ is attained by the function $Q_0$ which is the unique (up to symmetries) positive radial solution of	
	\begin{align}	
	-\Delta Q_0 + Q_0 = |Q_0|^{\alpha} Q_0. \label{ground state equation c=0}	
	\end{align}	
	In the case $c\ne 0$ and $c<\lambda(d)$, we have the following result (see \cite{KMVZ-focusing} and also \cite{Dinh}).	
	\begin{theorem}[Sharp Gagliardo-Nirenberg inequality] \label{theorem sharp gagliardo-nirenberg}		
		Let $d\geq 3$, $0<\alpha<\frac{4}{d-2}$ and $c\ne 0$ be such that $c<\lambda(d)$. Then $C_{\emph{GN}}(c) \in (0,\infty)$ and		
		\begin{itemize}			
			\item if $0<c<\lambda(d)$, then the equality in $(\ref{gagliardo nirenberg inequality})$ is attained by a function $Q_c \in H^1$ which is a positive radial solution to the elliptic equation			
			\begin{align}			
			-\Delta Q_c - c|x|^{-2} Q_c + Q_c = |Q_c|^{\alpha} Q_c. \label{ground state equation}			
			\end{align}			
			\item if $c<0$, then $C_{\emph{GN}}(c) = C_{\emph{GN}}(0)$ and the equality in $(\ref{radial gagliardo-nirenberg inequality})$ is never attained. However, the constant $C_{\emph{GN}}(c,\emph{rad})$ is attained by a function $Q_{c,\emph{rad}} \in H^1$ which is a positive solution to the elliptic equation			
			\begin{align}			-
			\Delta Q_{c,\emph{rad}} - c|x|^{-2} Q_{c,\emph{rad}} + Q_{c,\emph{rad}} = |Q_{c,\emph{rad}}|^{\alpha} Q_{c,\emph{rad}}. \label{radial ground state equation}			
			\end{align}			
		\end{itemize}		
	\end{theorem}	
	We refer the reader to \cite[Theorem 3.1]{KMVZ-focusing} for the proof in the case $d=3$ and $\alpha=2$ and to \cite[Theorem 4.1]{Dinh} for the proof in the general case.
	
	\subsection{Profile decomposition}
	We next recall the profile decomposition related to the nonlinear Schr\"odinger equation with inverse-square potential. 
	\begin{theorem}[Profile decomposition \cite{Bensouilah}] \label{theorem profile decomposition}
		Let $d\geq 3$ and $c \ne 0$ be such that $c<\lambda(d)$. Let $(v_n)_{n\geq 1}$ be a bounded sequence in $H^1$. Then there exist a subsequence still denoted by $(v_n)_{n\geq 1}$, a family $(x^j_n)_{n\geq 1}$ of sequences in $\R^d$ and a sequence $(V^j)_{j\geq 1}$ of $H^1$-functions such that
		\begin{enumerate}
			\item for every $j\ne k$, 
			\begin{align}
			|x^j_n-x^k_n| \rightarrow \infty, \label{pointwise orthogonality}
			\end{align}
			as $n\rightarrow \infty$;
			\item for every $l \geq 1$ and every $x \in \R^d$, we have
			\[
			v_n(x) = \sum_{j=1}^l V^j(x-x^j_n) + v^l_n(x), 
			\]
			with 
			\begin{align}
			\limsup_{n\rightarrow \infty} \|v^l_n\|_{L^p} \rightarrow 0, \label{profile error}
			\end{align}
			as $l \rightarrow \infty$ for any $2<p<\frac{2d}{d-2}$.
		\end{enumerate}
		Moreover, for every $l\geq 1$, we have
		\begin{align}
		\|v_n\|^2_{L^2} &=\sum_{j=1}^l \|V^j\|^2_{L^2} + \|v^l_n\|^2_{L^2} +o_n(1), \label{L2 expansion} \\
		\|\nabla v_n\|^2_{L^2} &= \sum_{j=1}^l \|\nabla V^j\|^2_{L^2} + \|\nabla v^l_n\|^2_{L^2} + o_n(1), \label{dot H1 expansion} \\
		\|v_n\|^2_{\dot{H}^1_c} &=\sum_{j=1}^l \|V^j(\cdot-x^j_n)\|^2_{\dot{H}^1_c} + \|v^l_n\|^2_{\dot{H}^1_c} +o_n(1), \label{hardy norm expansion} \\
		\|v_n\|^{\alpha+2}_{L^{\alpha+2}} &=\sum_{j=1}^l \|V^j\|^{\alpha+2}_{L^{\alpha+2}} + \|v^l_n\|^{\alpha+2}_{L^{\alpha+2}} +o_n(1), \label{lebesgue norm expansion}
		\end{align}
		for any $0<\alpha<\frac{4}{d-2}$.
	\end{theorem}
	We refer the reader to \cite[Theorem 4]{Bensouilah} for the proof of Theorem $\ref{theorem profile decomposition}$. We note here that the profile decomposition argument was first proposed by G\'erard in \cite{Gerard}. Later, Hmidi and Keraani \cite{HK} gave a refined version and used it to give a simple proof of some dynamical properties of blow-up solutions for the classical nonlinear Schr\"odinger equation. Using the same idea as in \cite{HK}, the profile decomposition of bounded sequences in $H^2$ and $H^s$ ($0<s<1$) were then established in \cite{ZZY, Zhu-blow-up-fnls} to study dynamical aspects of blow-up solutions for the fourth-order nonlinear Schr\"odinger equation and the fractional nonlinear Schr\"odinger equation. The profile decomposition was also successfully used to study the stability of standing waves for the fractional nonlinear Schr\"odinger equation (see e.g. \cite{Zhu, FZ, PS}).

	\section{Orbital stability of standing waves} \label{section stability}
	\setcounter{equation}{0}
	In this section, we will give the proof of the stability result stated in Theorem $\ref{theorem stability}$.
	Let us firstly study the variational problems $(\ref{variational problem 1})$ and $(\ref{variational problem 2})$ by using the profile decomposition of bounded sequences in $H^1$. 
	\begin{proposition} \label{proposition variational problems}
		Let $d\geq 3$, $0<\alpha<\frac{4}{d}$ and $M>0$. 
		\begin{enumerate}
			\item If $0<c<\lambda(d)$, then the variational problem $(\ref{variational problem 1})$ is well-defined and there exists $C_1>0$ such that 
			\begin{align}
			d_M \leq -C_1<0. \label{upper bound dM}
			\end{align}
			Moreover, there exists $v\in H^1$ such that $E(v) = d_M$.
			\item If $c<0$, then the variational problem $(\ref{variational problem 2})$ is well-defined and there exists $C_2>0$ such that
			\begin{align}
			d_{M,\emph{rad}} \leq -C_2<0. \label{upper bound dM radial}
			\end{align}
			Moreover, there exists $v\in H^1_{\emph{rad}}$ such that $E(v) = d_{M,\emph{rad}}$.
		\end{enumerate}
	\end{proposition}
	\begin{proof}
		(1) Let us firstly consider the case $0<c<\lambda(d)$. Let $v \in H^1$ be such that $\|v\|^2_{L^2} = M$. By the sharp Gagliardo-Nirenberg inequality $(\ref{gagliardo nirenberg inequality})$, we have
		\begin{align*}
		E(v)&=\frac{1}{2} \|v\|^2_{\dot{H}^1_c} - \frac{1}{\alpha+2} \|v\|^{\alpha+2}_{L^{\alpha+2}} \\
		&\geq \frac{1}{2} \|v\|^2_{\dot{H}^1_c} -\frac{C_{\text{GN}}(c)}{\alpha+2} \|v\|^{\frac{4-(d-2)\alpha}{2}}_{L^2} \|v\|^{\frac{d\alpha}{2}}_{\dot{H}^1_c} \\
		&= \frac{1}{2} \|v\|^2_{\dot{H}^1_c} -\frac{C_{\text{GN}}(c)}{\alpha+2} M^{\frac{4-(d-2)\alpha}{4}} \|v\|^{\frac{d\alpha}{2}}_{\dot{H}^1_c}.
		\end{align*}
		Since $0<\frac{d\alpha}{2}<2$, we apply the Young's inequality, that is for any $a,b>0$, any $\eps>0$ and any $1<p,q<\infty$ satisfying $\frac{1}{p}+\frac{1}{q}=1$, there exists $C(\eps, p,q)>0$ such that
		\[
		ab \leq \eps a^p + C(\eps, p,q) b^q,
		\]
		to have
		\[
		\frac{C_{\text{GN}}(c)}{\alpha+2} M^{\frac{4-(d-2)\alpha}{4}} \|v\|^{\frac{d\alpha}{2}}_{\dot{H}^1_c} \leq \eps \|v\|^2_{\dot{H}^1_c} + C(\eps, d, \alpha, M). 
		\]
		We thus get
		\begin{align}
		E(v) \geq \left(\frac{1}{2}-\eps\right) \|v\|^2_{\dot{H}^1_c} - C(\eps, d,\alpha, M). \label{lower bound energy}
		\end{align}
		By choosing $0<\eps<\frac{1}{2}$, we see that $E(v) \geq - C(\eps, d,\alpha,M)$. This shows that the variational problem $(\ref{variational problem 1})$ is well-defined. 
		
		For $\lambda>0$, we define $v_\lambda(x):= \lambda^{\frac{d}{2}} v(\lambda x)$. It is easy to check that $\|v_\lambda\|^2_{L^2} = \|v\|^2_{L^2} =M$ and
		\[
		E(v_\lambda) = \frac{\lambda^2}{2} \|v_\lambda\|^2_{\dot{H}^1_c} - \frac{\lambda^{\frac{d\alpha}{2}}}{\alpha+2} \|v\|^{\alpha+2}_{L^{\alpha+2}}.
		\]
		Since $0<\frac{d\alpha}{2}<2$, one can find a value $\lambda_1>0$ sufficiently small so that $E(v_{\lambda_1})<0$. Taking $C_1:= -E(v_{\lambda_1})>0$, one gets $(\ref{upper bound dM})$.
			
		Let $(v_n)_{n\geq 1}$ be a minimizing sequence of $d_M$, that is $\|v_n\|^2_{L^2} = M$ for all $n\geq 1$ and $\lim_{n\rightarrow \infty} E(v_n) = d_M$. There exists $C>0$ such that
		\[
		E(v_n) \leq d_M + C,
		\]
		for all $n \geq 1$. For $0<\eps<\frac{1}{2}$ fixed, we use $(\ref{lower bound energy})$ to infer
		\[
		\left(\frac{1}{2}-\eps \right) \|v_n\|^2_{\dot{H}^1_c} \leq d_M +C + C(\eps, d, \alpha, M).
		\]
		This shows that $\|v_n\|_{\dot{H}^1_c}$ (hence $\|v_n\|_{\dot{H}^1}$) is bounded for all $n\geq 1$. In particular $(v_n)_{n\geq 1}$ is a bounded sequence in $H^1$. Therefore, the profile decomposition given in Theorem $\ref{theorem profile decomposition}$ implies that up to a subsequence, we can write for every $l\geq 1$,
		\begin{align}
		v_n(x) = \sum_{j=1}^l V^j(x-x^j_n) +v^l_n(x), \label{profile decomposition}
		\end{align}
		and $(\ref{pointwise orthogonality})-(\ref{lebesgue norm expansion})$ hold. In particular, we have
		\begin{align}
		E(v_n) = \sum_{j=1}^l E(V^j(\cdot-x^j_n)) + E(v^l_n) + o_n(1), \label{energy expansion}
		\end{align}
		as $n\rightarrow \infty$. Denote $\tilde{V}^j_n(x):= \lambda_j V^j(x-x^j_n)$ and $\tilde{v}^l_n(x) = \lambda^l_n v^l_n(x)$, where
		\[
		\lambda_j := \frac{\sqrt{M}}{\|V^j\|_{L^2}} \geq 1, \quad \lambda_n^l:= \frac{\sqrt{M}}{\|v^l_n\|_{L^2}} \geq 1.
		\]
		We readily check that
		\[
		\|\tilde{V}^j_n\|^2_{L^2} = \lambda_j^2 \|V^j(\cdot- x^j_n)\|^2_{L^2} = M, \quad \|\tilde{v}^l_n\|^2_{L^2} = (\lambda^l_n)^2 \|v^l_n\|^2_{L^2} = M. 
		\]
		By definition of $d_M$, we have
		\begin{align}
		E(\tilde{V}^j_n) \geq d_M, \quad E(\tilde{v}^l_n) \geq d_M. \label{energy new variables}
		\end{align}
		Moreover, a direct computation shows that
		\begin{align*}
		E(\tilde{V}^j_n) = \frac{\lambda^2_j}{2} \|V^j(\cdot- x^j_n)\|^2_{\dot{H}^1_c} - \frac{\lambda_j^{\alpha+2}}{\alpha+2} \|V^j\|^{\alpha+2}_{L^{\alpha+2}}.
		\end{align*}
		So,
		\begin{align}
		E(V^j(\cdot-x^j_n)) = \frac{E(\tilde{V}^j_n)}{\lambda_j^2} + \frac{\lambda_j^\alpha-1}{\alpha+2} \|V^j\|^{\alpha+2}_{L^{\alpha+2}}. \label{energy main terms}
		\end{align}
		Similarly,
		\begin{align}
		E(v^l_n) = \frac{E(\tilde{v}^l_n)}{(\lambda^l_n)^2} + \frac{(\lambda^l_n)^\alpha-1}{\alpha+2} \|v^l_n\|^{\alpha+2}_{L^{\alpha+2}} \geq \frac{E(\tilde{v}^l_n)}{(\lambda^l_n)^2}. \label{energy remainder}
		\end{align}
		Inserting $(\ref{energy main terms})$, $(\ref{energy remainder})$ to $(\ref{energy expansion})$ and using $(\ref{energy new variables})$, we obtain
		\begin{align*}
		E(v_n) &\geq \sum_{j=1}^l \left( \frac{E(\tilde{V}^j_n)}{\lambda_j^2} + \frac{\lambda_j^\alpha-1}{\alpha+2} \|V^j\|^{\alpha+2}_{L^{\alpha+2}} \right) + \frac{E(\tilde{v}^l_n)}{(\lambda^l_n)^2} + o_n(1) \\
		& \geq \sum_{j=1}^l \frac{\|V^j\|^2}{M} d_M + \left( \inf_{j\geq 1} \frac{\lambda^\alpha_j -1}{\alpha+2} \right) \sum_{j=1}^l \|V^j\|^{\alpha+2}_{L^{\alpha+2}} + \frac{\|v^l_n\|^2_{L^2}}{M} d_M + o_n(1) \\
		& = \frac{d_M}{M} \left( \|v_n\|^2_{L^2} + o_n(1) \right) + \left( \inf_{j\geq 1} \frac{\lambda^\alpha_j -1}{\alpha+2} \right) \left( \|v_n\|^{\alpha+2}_{L^{\alpha+2}} - \|v^l_n\|^{\alpha+2}_{L^{\alpha+2}} + o_n(1)\right) + o_n(1).
		\end{align*}
		Since $\sum_{j=1}^\infty \|V^j\|^2_{L^2}$ is convergent, there exists $j_0 \geq 1$ such that
		\[
		\|V^{j_0}\|^2_{L^2} = \sup_{j\geq 1} \|V^j\|^2_{L^2}. 
		\]
		We thus have
		\begin{align}
		\inf_{j\geq 1} \lambda^\alpha_j= \inf_{j\geq 1} \left(\frac{\sqrt{M}}{\|V^j\|_{L^2} } \right)^\alpha = \left(\frac{\sqrt{M}}{\|V^{j_0}\|_{L^2} }\right)^\alpha. \label{infimum}
		\end{align}
		On the other hand, by the definition of energy,
		\[
		\frac{\|v_n\|^{\alpha+2}_{L^{\alpha+2}}}{\alpha+2} \geq -E(v_n) \nearrow -d_M.
		\]
		Using $(\ref{upper bound dM})$, we get
		\begin{align}
		\frac{\|v_n\|^{\alpha+2}_{L^{\alpha+2}}}{\alpha+2} \geq -E(v_n) \geq \frac{C_1}{2}, \label{lower bound}
		\end{align}
		for $n$ sufficiently large. By $(\ref{infimum})$ and $(\ref{lower bound})$, we get 
		\[
		E(v_n) \geq d_M + o_n(1) + \left[\left(\frac{\sqrt{M}}{\|V^{j_0}\|_{L^2} }\right)^\alpha -1 \right] \left(\frac{C_1}{2} - \frac{\|v^l_n\|^{\alpha+2}_{L^{\alpha+2}}}{\alpha+2} \right) + o_n(1).
		\]
		Taking the limits $n \rightarrow \infty$ and $l\rightarrow \infty$ and using $(\ref{profile error})$, we obtain
		\[
		d_M \geq d_M + \left[\left(\frac{\sqrt{M}}{\|V^{j_0}\|_{L^2} }\right)^\alpha -1 \right] \frac{C_1}{2}.
		\] 
		Therefore, $\|V^{j_0}\|^2_{L^2} \geq M$. The almost orthogonality $(\ref{L2 expansion})$ and the fact $\|v_n\|^2_{L^2}=M$ then imply that there is only one term $V^{j_0} \ne 0$ in the profile decomposition $(\ref{profile decomposition})$ and $\|V^{j_0}\|^2_{L^2} =M$. 
		
		The identity $(\ref{profile decomposition})$ implies for every $l \geq j_0$,
		\[
		v_n(x) = V^{j_0} (x-x^{j_0}_n) + v^l_n(x).
		\]
		Fix $l=j_0$. Using
		\[
		\|v_n\|^2_{L^2} = \|V^{j_0}\|^2_{L^2} + \|v^{j_0}_n\|^2_{L^2} + o_n(1),
		\]
		as $n\rightarrow \infty$, and $\|v_n\|^2_{L^2} = \|V^{j_0}\|^2_{L^2} = M$, we get (up to a subsequence)
		\[
		\lim_{n\rightarrow \infty} \|v^{j_0}_n\|_{L^2} =0.
		\]
		This shows that the sequence $(v^{j_0}_n)_{n\geq 1}$ converges to zero weakly in $H^1$ and strongly in $L^2$. The boundedness of $(v^{j_0}_n)_{n\geq 1}$ in $H^1$ along with the strong convergence in $L^2$ to zero yield that
		\[
		\lim_{n\rightarrow \infty} \|v^{j_0}_n\|^{\alpha+2}_{L^{\alpha+2}} =0.
		\]
		The lower semi-continuity of Hardy's functional then gives
		\[
		0 \leq \liminf_{n\rightarrow \infty} E(v^{j_0}_n),
		\]
		thus
		\begin{align*}
		\liminf_{n\rightarrow \infty} E(V^{j_0}(\cdot-x^{j_0}_n)) &\leq \liminf_{n\rightarrow \infty} E(V^{j_0}(\cdot-x^{j_0}_n)) + \liminf_{n\rightarrow \infty} E(v^{j_0}_n) \\
		&\leq \liminf_{n\rightarrow \infty} \left(E(V^{j_0}(\cdot-x^{j_0}_n)) + E(v^{j_0}_n) \right) \\
		&=\liminf_{n\rightarrow \infty} E(v_n) = d_M.
		\end{align*}
		On the other hand, since $\|V^{j_0}(\cdot-x^{j_0}_n)\|^2_{L^2} = \|V^{j_0}\|^2_{L^2} = M$ for all $n\geq 1$, we have $E(V^{j_0}(\cdot-x^{j_0}_n)) \geq d_M$ for all $n\geq 1$. Therefore,
		\[
		\liminf_{n\rightarrow \infty} E(V^{j_0}(\cdot-x^{j_0}_n)) = d_M,
		\]
or equivalently,
		\begin{align}
		\frac{1}{2} \|\nabla V^{j_0}\|^2_{L^2} - \frac{1}{\alpha+2} \|V^{j_0}\|^{\alpha+2}_{L^{\alpha+2}} -\frac{c}{2} \limsup_{n\rightarrow \infty} \int |x|^{-2} |V^{j_0}(x-x^{j_0}_n)|^2 dx = d_M. \label{liminf equality}
		\end{align}
		We next prove that the sequence $(x^{j_0}_n)_{n\geq 1}$ is bounded. Indeed, if it is not true, then up to a subsequence, we assume that $|x^{j_0}_n| \rightarrow \infty$ as $n\rightarrow \infty$. Without loss of generality, we assume that $V^{j_0}$ is continuous and compactly suported. We have
		\[
		\int |x|^{-2} |V^{j_0}(x-x^{j_0}_n)|^2 dx = \int_{\text{supp}(V^{j_0})} |x+x^{j_0}_n|^{-2} |V^{j_0}(x)|^2 dx.
		\]
		Since $|x^{j_0}_n| \rightarrow \infty$ as $n\rightarrow \infty$, we see that $|x+ x^{j_0}_n| \geq |x^{j_0}_n| -|x| \rightarrow \infty$ as $n\rightarrow \infty$ for all $x \in \text{supp}(V^{j_0})$. This shows that
		\[
		\int |x|^{-2} |V^{j_0}(x-x^{j_0}_n)|^2 dx \rightarrow 0,
		\]
		as $n\rightarrow \infty$. This yields
		\[
		\frac{1}{2} \|\nabla V^{j_0}\|^2_{L^2} - \frac{1}{\alpha+2} \|V^{j_0}\|^{\alpha+2}_{L^{\alpha+2}} = d_M.
		\]
		By the definition of $E(V^{j_0})$, we obtain
		\[
		E(V^{j_0}) + \frac{c}{2} \int |x|^{-2} |V^{j_0}(x)|^2 dx = d_M.
		\]
		Since $0<c<\lambda(d)$, we get $E(V^{j_0}) <d_M$, which is absurd (note that $E(V^{j_0}) \geq d_M$ due to $\|V^{j_0}\|^2_{L^2}=M$). Therefore, the sequence $(x^{j_0}_n)_{n\geq 1}$ is bounded and up to a subsequence, we assume that $x^{j_0}_n \rightarrow x^{j_0}$ as $n\rightarrow \infty$.  
		
		We now write
		\[
		v_n(x) = \tilde{V}^{j_0}(x) + \tilde{v}^{j_0}_n(x),
		\]
		where $\tilde{V}^{j_0}(x)= V^{j_0}(x-x^{j_0})$ and $\tilde{v}^{j_0}_n(x) := V^{j_0}(x-x^{j_0}_n) - V^{j_0}(x-x^{j_0}) + v^{j_0}_n(x)$. Using the fact $\|v_n\|^2_{L^2} = \|V^{j_0}\|^2_{L^2} = M$, it is easy to see that
		\[
		\tilde{v}^{j_0}_n \rightharpoonup 0 \text{ weakly in } H^1 \text{ and } \lim_{n\rightarrow \infty} \|\tilde{v}^{j_0}_n\|_{L^2} =0.
		\]
		The first observation on $\tilde{v}^{j_0}_n$ allows us to write
		\[
		E(v_n) = E(\tilde{V}^{j_0}) + E(\tilde{v}^{j_0}_n) + o_n(1).
		\]
		Again, the lower semi-continuity of Hardy's functional and the fact $\lim_{n\rightarrow \infty} \|\tilde{v}^{j_0}_n\|^{\alpha+2}_{L^{\alpha+2}} =0$, we get that $\liminf_{n\rightarrow \infty} E(\tilde{v}^{j_0}_n) \geq 0$. Hence, using the fact that $\|\tilde{V}^{j_0}\|^2_{L^2} = M$, we infer that
		\begin{align*}
		d_M =\liminf_{n\rightarrow \infty} E(v_n) &\geq \liminf_{n\rightarrow \infty} \left( E(\tilde{V}^{j_0}) + E(\tilde{v}^{j_0}_n) \right) \\
		& \geq E(\tilde{V}^{j_0}) + \liminf_{n\rightarrow \infty} E(\tilde{v}^{j_0}_n) \\
		&\geq E(\tilde{V}^{j_0}) \geq d_M.
		\end{align*}
		Therefore, $E(\tilde{V}^{j_0}) = d_M$ which completes the proof of Item (1).
		
		(2) We now consider the case $c<0$. By the same argument (with $C_{\text{GN}}(c,\text{rad})$ in place of $C_{\text{GN}}(c)$), the variational problem $(\ref{variational problem 2})$ is well-defined and there exists $C_2>0$ such that $(\ref{upper bound dM radial})$ holds. It remains to show that there exists $v\in H^1$ radial such that $E(v) = d_{M,\text{rad}}$. Let $(v_n)_{n\geq 1}$ be a minimizing sequence of $d_{M,\text{rad}}$, that is $v_n \in H^1_{\text{rad}}$, $\|v_n\|^2_{L^2} = M$ for all $n\geq 1$ and $\lim_{n\rightarrow \infty} E(v_n) = d_{M,\text{rad}}$. Arguing as in the first case, we see that $(v_n)_{n\geq 1}$ is a bounded radial sequence in $H^1$. Thanks to the fact that 
		\begin{align}
		H^1_{\text{rad}} \hookrightarrow L^p \text{ compactly}, \label{compact embedding}
		\end{align}
		for any $2<p< \frac{2d}{d-2}$, there exists $V \in H^1_{\text{rad}}$ (see Appendix) such that
		\[
		v_n \rightharpoonup V \text{ weakly in } H^1 \text{ and } v_n \rightarrow V \text{ strongly in } L^{\alpha+2}.
		\]
		We write
		\[
		v_n(x) = V(x) + r_n(x), 
		\]
		with $r_n \rightharpoonup 0$ weakly in $H^1$ (note that $r_n$ can be taken radially symmetric). We have the following expansions
		\begin{align}
		\|v_n\|^2_{L^2} &= \|V\|^2_{L^2} + \|r_n\|^2_{L^2} + o_n(1), \label{mass expansion radial} \\
		E(v_n) &= E(V) + E(r_n) + o_n(1), \label{energy expansion radial}
		\end{align}
		as $n\rightarrow \infty$. Denote $\tilde{V} = \lambda V$ and $\tilde{r}_n = \lambda_n r_n$, where
		\[
		\lambda:= \frac{\sqrt{M}}{\|V\|_{L^2}} \geq 1, \quad \lambda_n := \frac{\sqrt{M}}{\|r_n\|_{L^2}} \geq 1.
		\]
		It is obvious that $\|\tilde{V}\|^2_{L^2} = \|\tilde{r}_n\|^2_{L^2} = M$, hence 
		\[
		E(\tilde{V}) \geq d_{M,\text{rad}}, \quad E(\tilde{r}_n) \geq d_{M,\text{rad}}.
		\]
		We also have
		\[
		E(\tilde{V}) = \frac{\lambda^2}{2} \|V\|^2_{\dot{H}^1_c} - \frac{\lambda^{\alpha+2}}{\alpha+2} \|V\|^{\alpha+2}_{L^{\alpha+2}}.
		\]
		So,
		\[
		E(V) = \frac{E(\tilde{V})}{\lambda^2} + \frac{\lambda^\alpha-1}{\alpha+2} \|V\|^{\alpha+2}_{L^{\alpha+2}}.
		\]
		Similarly,
		\[
		E(r_n) = \frac{E(\tilde{r}_n)}{\lambda_n^2} + \frac{\lambda_n^\alpha-1}{\alpha+2} \|r_n\|^{\alpha+2}_{L^{\alpha+2}} \geq \frac{E(\tilde{r}_n)}{\lambda_n^2}.
		\]
		Pluging above estimates to $(\ref{energy expansion radial})$, we have
		\begin{align*}
		E(v_n) &\geq \frac{E(\tilde{V})}{\lambda^2} +\frac{\lambda^\alpha-1}{\alpha+2} \|V\|^{\alpha+2}_{L^{\alpha+2}} + \frac{E(\tilde{r}_n)}{\alpha_n^2} + o_n(1) \\
		&= \frac{\|V\|^2_{L^2} d_{M,\text{rad}}}{M} + \frac{\lambda^\alpha-1}{\alpha+2} \|V\|^{\alpha+2}_{L^{\alpha+2}} + \frac{\|r_n\|^2_{L^2} d_{M,\text{rad}}}{M} + o_n(1) \\
		&= \frac{d_{M,\text{rad}}}{M} \left(\|V\|^2_{L^2} + \|r_n\|^2_{L^2} \right) +\frac{\lambda^\alpha-1}{\alpha+2} \left( \|v_n\|^{\alpha+2}_{L^{\alpha+2}} - \|r_n\|^{\alpha+2}_{L^{\alpha+2}} \right) + o_n(1).
		\end{align*}
		Since 
		\[
		\frac{\|v_n\|^{\alpha+2}_{L^{\alpha+2}}}{\alpha+2} \geq -E(v_n) \nearrow -d_{M,\text{rad}}.
		\]
		Using the upper bound $(\ref{upper bound dM radial})$, we see that
		\[
		\frac{\|v_n\|^{\alpha+2}_{L^{\alpha+2}}}{\alpha+2} \geq -E(v_n) \geq \frac{C_2}{2},
		\]
		for $n$ sufficiently large. Taking $n\rightarrow \infty$, this combined with the fact $\lim_{n\rightarrow \infty} \|r_n\|^{\alpha+2}_{L^{\alpha+2}} =0$ yield
		\[
		d_{M,\text{rad}} \geq d_{M,\text{rad}} + \left[ \left(\frac{\sqrt{M}}{\|V\|_{L^2}} \right)^\alpha-1 \right]\frac{C_2}{2}.
		\]
		We thus obtain $\|V\|^2_{L^2} \geq M$. Since $\|v_n\|^2_{L^2} = M$, we have from $(\ref{mass expansion radial})$ that $\|V\|^2_{L^2} = \|v_n\|^2_{L^2} = M$. In particular, we have $\lim_{n\rightarrow\infty} \|r_n\|^2_{L^2}=0$ and $E(V) \geq d_{M,\text{rad}}$. Since $r_n\rightharpoonup 0$ weakly in $H^1$ and strongly in $L^2$. The lower semi-continuity of Hardy's functional implies
		\[
		\liminf_{n\rightarrow \infty} E(r_n) \geq 0.
		\]
		Therefore,
		\begin{align*}
		d_{M,\text{rad}} = \liminf_{n\rightarrow \infty} E(v_n) \geq \liminf_{n\rightarrow \infty} \left(E(V) + E(r_n) \right) &\geq E(V) + \liminf_{n\rightarrow \infty} E(r_n) \\
		&\geq E(V) \geq d_{M,\text{rad}}.
		\end{align*}
		We thus obtain $E(V) = d_{M,\text{rad}}$, which implies that the variational problem $(\ref{variational problem 2})$ is attained. The proof is complete.
	\end{proof}
	
	\begin{remark} \label{remark variational problems}
		(1) The proof Proposition $\ref{proposition variational problems}$ still holds 
		
		$\bullet$ in the case $0<c<\lambda(d)$ if in place of 
			\[
			\|v_n\|^2_{L^2} =M \text{ for all } n\geq 1 \text{ and } \lim_{n \rightarrow \infty} E(v_n) =d_M,
			\]
			we assume that
			\[
			\lim_{n\rightarrow \infty} \|v_n\|^2_{L^2} = M \text{ and } \lim_{n\rightarrow \infty} E(v_n) =d_M.
			\]
			
		$\bullet$ in the case $c<0$ if in place of 
			\[
			\|v_n\|^2_{L^2} =M \text{ for all } n\geq 1 \text{ and } \lim_{n \rightarrow \infty} E(v_n) =d_{M,\text{rad}},
			\]
			we assume that
			\[
			\lim_{n\rightarrow \infty} \|v_n\|^2_{L^2} = M \text{ and } \lim_{n\rightarrow \infty} E(v_n) =d_{M,\text{rad}}.
			\]
		(2) It follows from the proof of Proposition $\ref{proposition variational problems}$ that
		
			$\bullet$ in the case $0<c<\lambda(d)$,		
			\[
			E(v_n) = E(\tilde{V}^{j_0}) + E(\tilde{v}^{j_0}_n) + o_n(1),
			\]
			as $n\rightarrow \infty$, and
			\begin{align*}
			\lim_{n\rightarrow \infty} E(v_n) =d_M, \quad \lim_{n\rightarrow \infty} \|\tilde{v}^{j_0}_n\|^{\alpha+2}_{L^{\alpha+2}} =0, \quad E(\tilde{V}^{j_0}) = d_M.
			\end{align*}
			We thus have that up to a subsequence,
			\[
			\lim_{n\rightarrow \infty} \|\tilde{v}^{j_0}_n\|^2_{\dot{H}^1_c} =0.
			\]
			Since $\|\tilde{v}^{j_0}_n\|^2_{\dot{H}^1_c}  \sim \|\tilde{v}^{j_0}_n\|^2_{\dot{H}^1}$, we conclude that
			\[
			\lim_{n\rightarrow \infty} \|\nabla \tilde{v}^{j_0}_n\|_{L^2} =0.
			\] 
			So,
			\[
			\lim_{n\rightarrow \infty} \|\nabla v_n\|_{L^2} = \|\nabla \tilde{V}^{j_0}\|_{L^2},
			\]
			which along with $\lim_{n\rightarrow \infty} \|v_n\|_{L^2} = \|\tilde{V}^{j_0}\|_{L^2}$ yield
			\[
			v_n \rightarrow \tilde{V}^{j_0} \text{ strongly in } H^1,
			\]
			as $n\rightarrow \infty$. 
			
			$\bullet$ in the case $c<0$,
			\[
			E(v_n) = E(V) + E(r_n) +o_n(1),
			\]
			as $n\rightarrow \infty$, and 
			\[
			\lim_{n\rightarrow \infty} E(v_n) = d_{M,\text{rad}}, \quad \lim_{n\rightarrow \infty} \|r_n\|^{\alpha+2}_{L^{\alpha+2}} =0, \quad E(V) = d_{M,\text{rad}}.
			\]
			We thus have
			\[
			\lim_{n\rightarrow \infty} \|r_n\|^2_{\dot{H}^1_c} =0,
			\]
			which together with $\|r_n\|^2_{\dot{H}^1_c} \sim \|r_n\|^2_{\dot{H}^1}$ yield
			\[
			\lim_{n\rightarrow \infty} \|\nabla r_n\|_{L^2}=0.
			\]
			This implies
			\[
			\lim_{n\rightarrow \infty} \|\nabla v_n\|_{L^2} = \|\nabla V\|_{L^2},
			\]
			which together with $\lim_{n\rightarrow \infty} \|v_n\|_{L^2} = \|V\|_{L^2}$ imply
			\[
			v_n \rightarrow V \text{ strongly in } H^1,
			\]
			as $n \rightarrow \infty$.
	
	\end{remark}
	
	We are now able to prove the orbital stability given in Theorem $\ref{theorem stability}$.
	
	\noindent {\it Proof of Theorem $\ref{theorem stability}$.}
	We only consider the case $0<c<\lambda(d)$, the case $c<0$ is completely similar. We argue by contradiction. Assume that there exist sequences $(u_{0,n})_{n\geq 1} \subset H^1$, $(t_n)_{n\geq 1} \subset \R^+$ and $\eps_0>0$ such that for all $n\geq 1$,
	\begin{align}
	\inf_{v \in S_M} \|u_{0,n} - v\|_{H^1} <\frac{1}{n}, \label{stability proof 1}
	\end{align}
	and
	\begin{align}
	\inf_{v \in S_M} \|u_n(t_n) - v\|_{H^1} \geq \eps_0. \label{stability proof 2}
	\end{align}
	Here $u_n(t)$ is the solution to $(\ref{focusing NLS inverse square})$ with initial data $u_{0,n}$. We next claim that there exists $v \in S_M$ such that 
	\[
	\lim_{n\rightarrow \infty} \|u_{0,n} -v\|_{H^1} =0.
	\]
	Indeed, we have from $(\ref{stability proof 1})$ that for each $n\geq 1$, there exists $v_n \in S_M$ such that 
	\begin{align}
	\|u_{0,n} - v_n\|_{H^1} <\frac{2}{n}. \label{claim proof 1}
	\end{align}
	We thus obtain a sequence $(v_n)_{n\geq 1} \subset S_M$, and we get from the proof of Proposition $\ref{proposition variational problems}$ that there exists $v\in S_M$ such that
	\begin{align}
	\lim_{n\rightarrow \infty} \|v_n - v\|_{H^1} =0. \label{claim proof 2}
	\end{align}
	The claim follows immediately from $(\ref{claim proof 1})$ and $(\ref{claim proof 2})$. We thus get
	\[
	\lim_{n\rightarrow \infty} \|u_{0,n}\|^2_{L^2} = \|v\|_{L^2}^2 = M, \quad \lim_{n\rightarrow \infty} E(u_{0,n}) = E(v) = d_M.
	\]
	The conservation of mass and energy then imply
	\[
	\lim_{n\rightarrow \infty} \|u_n(t_n)\|^2_{L^2} = M, \quad \lim_{n\rightarrow \infty} E(u_n(t_n)) = d_M.
	\]
	Again, from the proof of Proposition $\ref{proposition variational problems}$ and Remark $\ref{remark variational problems}$, there exists $\tilde{v} \in S_M$ such that $(u_n(t_n))_{n\geq 1}$ converges strongly to $\tilde{v}$ in $H^1$. This contradicts $(\ref{stability proof 2})$ and the proof is complete.
	\defendproof

	\section{Strong instability of standing waves} \label{section instability}
	\setcounter{equation}{0}
	In this section, we study the stability of standing waves for the nonlinear Schr\"odinger equation with inverse-square potential in the $L^2$-critical case, i.e. $\alpha=\frac{4}{d}$ in $(\ref{focusing NLS inverse square})$. Let us start by defining properly the notion of ground states related to the $L^2$-critical $(\ref{focusing NLS inverse square})$.
	\begin{definition}[Ground states]
		\begin{itemize}
			\item In the case $0<c<\lambda(d)$, we call {\bf ground states} the maximizers of $J^{4/d}_c$ (see $(\ref{weinstein functional})$) which are positive radial solutions of
			\begin{align}
			-\Delta Q - c|x|^{-2} Q + Q = |Q|^{\frac{4}{d}} Q. \label{ground state positive c}
			\end{align} 
			The set of ground states is denoted by $\mathcal{G}$. 
			\item In the case $c<0$, we call {\bf radial ground states} the maximizers of $J^{4/d}_c$ which are positive radial solutions of
			\begin{align}
			-\Delta Q_{\text{rad}} - c|x|^{-2} Q_{\text{rad}} + Q_{\text{rad}} = |Q_{\text{rad}}|^{\frac{4}{d}} Q_{\text{rad}}. \label{ground state negative c}
			\end{align}
			The set of radial ground states is denoted by $\mathcal{G}_{\text{rad}}$.
		\end{itemize}
	\end{definition}
	
	\begin{remark} \label{remark instability}
		\begin{itemize}
			\item This notion of ground states in the case $0<c<\lambda(d)$ was first introduced in \cite{CG} due to the fact that the uniqueness (up to symmetries) of positive radial solutions of $(\ref{ground state positive c})$ and $(\ref{ground state negative c})$ are not yet known.
			\item It follows from Theorem $\ref{theorem sharp gagliardo-nirenberg}$ and its proof that there exists $M_{\text{gs}}>0$ such that $\|Q\|_{L^2} = M_{\text{gs}}$ for all $Q \in \mathcal{G}$. The constant $M_{\text{gs}}$ is called {\bf minimal mass}. Similarly, there exists $M_{\text{gs,rad}}>0$ such that $\|Q_{\text{rad}}\|_{L^2} = M_{\text{gs, rad}}$ for all $Q_{\text{rad}} \in \mathcal{G}_{\text{rad}}$. The constant $M_{\text{gs,rad}}$ is called {\bf radial minimal mass}.
			\item Thanks to the pseudo-conformal invariance, it was shown in \cite[Remark 3, p.118]{CG} that any solution $Q$ of $(\ref{ground state positive c})$ satisfies $\|Q\|_{L^2} \geq M_{\text{gs}}$. Similarly, any radial solution $Q_{\text{rad}}$ of $(\ref{ground state negative c})$ satisfies $\|Q_{\text{rad}}\|_{L^2} \geq M_{\text{gs,rad}}$.
			\item If $Q \in \mathcal{G}$, then $Q_\omega:= (\sqrt{\omega})^{d/2} Q(\sqrt{\omega} x)$ is also a maximizer of $J^{4/d}_c$ satisfying $\|Q_\omega\|_{L^2} = \|Q\|_{L^2} = M_{\text{gs}}$ and 
			\[
			-\Delta Q_\omega - c|x|^{-2} Q_\omega + \omega Q_\omega = |Q_\omega|^{\frac{4}{d}} Q_\omega.
			\]
			Similarly, if $Q_{\text{rad}} \in \mathcal{G}_{\text{rad}}$, then $Q_{\omega,\text{rad}} := (\sqrt{\omega})^{d/2} Q_{\text{rad}}(\sqrt{\omega} x)$ is aslo a maximizer of $J^{4/d}_c$ satisfying $\|Q_{\omega,\text{rad}}\|_{L^2} = \|Q_{\text{rad}}\|_{L^2} = M_{\text{gs,rad}}$ and
			\[
			-\Delta Q_{\omega,\text{rad}} - c|x|^{-2} Q_{\omega,\text{rad}} + \omega Q_{\omega,\text{rad}} = |Q_{\omega,\text{rad}}|^{\frac{4}{d}} Q_{\omega,\text{rad}}.
			\]
			\item Let $Q$ be a $H^1$-solution to $(\ref{ground state positive c})$. Multiplying both sides of $(\ref{ground state positive c})$ with $\overline{Q}$, integrating over $\R^d$ and using the integration by parts, we obtain
			\begin{align}
			\|Q\|^2_{\dot{H}^1_c} + \|Q\|^2_{L^2} = \|Q\|^{\frac{4}{d}+2}_{L^{\frac{4}{d}+2}}. \label{pohozaev identity 1}
			\end{align}
			Similarly, multiplying both sides of $(\ref{ground state positive c})$ with $x\cdot \nabla \overline{Q}$, integrating over $\R^d$, the integration by parts gives
			\begin{align}
			\frac{d-2}{2} \|Q\|^2_{\dot{H}^1_c} + \frac{d}{2} \|Q\|^2_{L^2} = \frac{d^2}{2d+4} \|Q\|^{\frac{4}{d}+2}_{L^{\frac{4}{d}+2}}. \label{pohozaev identity 2}
			\end{align}
			Combining $(\ref{pohozaev identity 1})$ and $(\ref{pohozaev identity 2})$, we obtain the following Pohozaev's identities:
			\begin{align}
			\|Q\|^2_{\dot{H}^1_c}= \frac{d}{d+2}\|Q\|^{\frac{4}{d}+2}_{L^{\frac{4}{d}+2}} =\frac{d}{2}\|Q\|^2_{L^2}. \label{pohozaev identity}
			\end{align}
			In particular, we have
			\begin{align}
			E(Q) = \frac{1}{2} \|Q\|^2_{\dot{H}^1_c} -\frac{d}{2d+4} \|Q\|^{\frac{4}{d}+2}_{L^{\frac{4}{d}+2}}= 0. \label{energy ground state}
			\end{align}
			Similar identities as $(\ref{pohozaev identity 1}) - (\ref{energy ground state})$ still hold if $Q$ is replaced by $Q_{\text{rad}}$ which is a $H^1$-solution of $(\ref{ground state negative c})$.
		\end{itemize}
	\end{remark}
    \noindent \textit{Proof of Theorem $\ref{theorem instability I}$.}
		The proof is based on a standard argument (see e.g. \cite[Theorem 8.2.1]{Cazenave}). Let us assume at the moment that the functions $Q$ and $Q_{\text{rad}}$ belong to $L^2(|x|^2 dx)$. Note that unlike the classical nonlinear Schr\"odinger equation, we do not know whether solutions of $(\ref{ground state positive c})$ and $(\ref{ground state negative c})$ enjoy the exponential decay at infinity or not. In the case the exponential decay at infinity holds true, the above assumption is obviously satisfied. The above assumption together with the Pohozaev identity $(\ref{pohozaev identity})$ imply $Q$ and $Q_{\text{rad}}$ are both in $H^1 \cap L^2(|x|^2 dx)$. We next recall the standard virial identity related to the $L^2$-critical $(\ref{focusing NLS inverse square})$ (see \cite[Lemma 5.3]{Dinh} or \cite[Lemma 3, p.124]{CG}). 
		\begin{lemma}[Virial identity] \label{lemma virial identity}
			Let $d\geq 3$, $c\ne 0$ and $c<\lambda(d)$. Let $u_0 \in H^1 \cap L^2(|x|^2 dx)$ and $u: J \times \R^d \rightarrow \C$ the corresponding solution to the $L^2$-critical $(\ref{focusing NLS inverse square})$. Then $u \in C(J, H^1 \cap L^2(|x|^2 dx))$ and
			\[
			\frac{d^2}{dt^2} \|x u(t)\|^2_{L^2}= 16 E(u_0),
			\]
			for any $t\in J$.
		\end{lemma}
		We now denote for $0<c<\lambda(d)$,
		\[
		u_{0,n}(x) := \mu_n Q(x),
		\]
		and for $c<0$,
		\[
		u_{0,n}(x):= \mu_n Q_{\text{rad}}(x),
		\]
		for any $n\geq 1$, where $\mu_n:= 1+1/n$. It is obvious that for $0<c<\lambda(d)$, $u_{0,n} \rightarrow Q$ strongly in $H^1$, and for $c<0$, $u_{0,n} \rightarrow Q_{\text{rad}}$ strongly in $H^1$. Let $u_n$ be the corresponding solution to the $L^2$-critical $(\ref{focusing NLS inverse square})$ with initial data $u_{0,n}$. We will show that $u_n$ blows up in finite time for any $n\geq 1$. We only consider $Q$, the one for $Q_{\text{rad}}$ is similar. To see this, we fix $n\geq 1$ and compute
		\begin{align*}
		E(u_{0,n}) &= \frac{\mu_n^2}{2} \|Q\|^2_{\dot{H}^1_c} -\frac{d\mu_n^{\frac{4}{d}+2}}{2d+4} \|Q\|^{\frac{4}{d}+2}_{L^{\frac{4}{d}+2}} \\
		&= \mu_n^2 E(Q) + \frac{d}{2d+4} \mu_n^2 \left(1- \mu_n^{\frac{4}{d}} \right) \|Q\|^{\frac{4}{d}+2}_{L^{\frac{4}{d}+2}}.
		\end{align*}
		Since $E(Q) =0$ (see $(\ref{energy ground state})$) and $\mu_n>1$, we have that $E(u_{0,n}) <0$. By Lemma $\ref{lemma virial identity}$, we have
		\[
		\frac{d^2}{dt^2} E(u_n(t)) = 16 E(u_{0,n}) <0,
		\]
		for any $t$ as long as the solution exists. The standard convexity argument (see e.g. \cite{Glassey}) shows that $u_n$ must blow up in finite time. 
				
		It remains to show that $Q$ and $Q_{\text{rad}}$ belongs to $L^2(|x|^2 dx)$. Let us first consider $Q$. Denote $u(t,x):= e^{it} Q(x)$ the standing waves. It is easy to see that $u$ is a global solution of the $L^2$-critical $(\ref{focusing NLS inverse square})$. For $0<T<+\infty$, we denote
		\[
		u_T(t,x) := \frac{e^{-i\frac{|x|^2}{4(T-t)}}}{(T-t)^{\frac{d}{2}}} u\left(\frac{1}{T-t}, \frac{x}{T-t} \right).
		\]
		Since the $L^2$-critical $(\ref{focusing NLS inverse square})$ is invariant under the pseudo-conformal transformation, we have from \cite[Lemma 1, p.117]{CG} that $u_T$ is a solution of the $L^2$-critical $(\ref{focusing NLS inverse square})$ which blows up at $T$ and satisfies $\|u_T(t)\|_{L^2} =\|u(1/(T-t))\|_{L^2}$. We thus construct a solution to the $L^2$-critical $(\ref{focusing NLS inverse square})$ which blows up in finite time $T$ and its initial data satisfies
		\[
		\|u_T(0)\|_{L^2} = \|Q\|_{L^2}= M_{\text{gs}}.
		\]
		We have from \cite[Theorem 3.2]{BD} that for a time sequence $t_n \nearrow T$ as $n\rightarrow \infty$, there exists $\tilde{Q} \in \mathcal{G}$, sequences of $\theta_n \in \R$, $\lambda_n>0$ and $x_n \in \R^d$ such that
		\begin{align}
		e^{it\theta_n} \lambda_n^{\frac{d}{2}} u_T(t_n, \lambda_n \cdot+ x_n) \rightarrow \tilde{Q} \text{ strongly in } H^1, \label{limiting profile}
		\end{align}
		as $n \rightarrow \infty$. We aslo have from \cite[p.127-128]{CG} that $u_T(t) \in L^2(|x|^2 dx)$ for any $t\in [0,T)$. In particular, $Q \in L^2(|x|^2 dx)$. This completes the proof for $Q$. 
		
		The case for $Q_{\text{rad}}$ is similar. The only different point is that instead of $(\ref{limiting profile})$, we have
		\[
		e^{it\vartheta_n} \rho_n^{\frac{d}{2}} u_T(t_n, \rho_n \cdot) \rightarrow \tilde{Q}_{\text{rad}} \text{ strongly in } H^1,
		\]
		as $n\rightarrow \infty$, for some $\tilde{Q}_{\text{rad}} \in \mathcal{G}_{\text{rad}}$. The rest of the proof remains the same as for $Q$. The proof is complete.		
	\defendproof

	\begin{theorem} [Strong instability II] \label{theorem instability II}
		Let $d\geq 3$ and $c \ne 0$ be such that $c \leq \lambda(d)$. Let $\omega>0$ and $Q$ be a radial $H^1$-solution to the elliptic equation
		\begin{align}
			-\Delta Q - c|x|^{-2} Q + \omega Q = |Q|^{\frac{4}{d}} Q. \label{mass-critical elliptic equation}
		\end{align}
		Then the standing wave $e^{i\omega t} Q(x)$ is unstable in the following sense: there exists $(u_{0,n})_{n\geq 1} \subset H^1$ such that
		\[
		u_{0,n} \rightarrow Q \text{ strongly in } H^1,
		\]
		as $n\rightarrow \infty$ and the corresponding solution $u_n$ to the $L^2$-critical $(\ref{focusing NLS inverse square})$ with initial data $u_{0,n}$ blows up in finite time for any $n\geq 1$.
	\end{theorem}
	
	\begin{remark} \label{remark instability omega}
		\begin{itemize}
			\item The strong instability of Theorem $\ref{theorem instability II}$ allows radial solutions of $(\ref{ground state positive c})$ and $(\ref{ground state negative c})$ whose the $L^2$-norms may larger than $M_{\text{gs}}$ and $M_{\text{gs,rad}}$.
			\item As in Remark $\ref{remark instability}$, if $Q$ is a $H^1$-solution to $(\ref{mass-critical elliptic equation})$, then we have the following Pohozaev identities 
			\begin{align}
			\|Q\|^2_{\dot{H}^1_c}= \frac{d}{d+2}\|Q\|^{\frac{4}{d}+2}_{L^{\frac{4}{d}+2}} =\frac{d\omega}{2}\|Q\|^2_{L^2}. \label{pohozaev identity omega}
			\end{align}
		\end{itemize}
	\end{remark}
	
	In order to show Theorem $\ref{theorem instability II}$, we recall the following virial estimates related to the $L^2$-critical $(\ref{focusing NLS inverse square})$. Let $\theta:[0,\infty) \rightarrow [0,\infty)$ be  a function satisfying
	\begin{align}
		\theta(r) = \left\{
		\begin{array}{cl}
			r^2 &\text{if } 0 \leq r \leq 1, \\
			\text{const.} &\text{if } r \geq 2, 
		\end{array}
		\right.
		\quad \text{and} \quad \theta''(r) \leq 2 \text{ for all } r\geq 0. \label{choice of theta} 
	\end{align}
	Note that the precise constant in $(\ref{choice of theta})$ is not important here. For $R>1$, we define the radial function
	\begin{align}
		\varphi_R(x) = \varphi_R(r) := R^2 \theta(r/R), \quad r=|x|. \label{define varphi_R}
	\end{align}
	We readily see that
	\[
	2-\varphi''_R(r) \geq 0, 2-\frac{\varphi'_R(r)}{r} \geq 0, \quad 2d -\Delta \varphi_R (x) \geq 0, \quad \forall r \geq 0, \quad \forall x \in \R^d.
	\]
	Let $u$ be a solution to $(\ref{focusing NLS inverse square})$. We define the localized virial potential associated to $u$ by
	\[
	V_{\varphi_R}(u(t)):= \int \varphi_R(x) |u(t,x)|^2 dx.
	\]
	\begin{lemma}[Radial virial estimate \cite{Dinh}] \label{lemma radial virial estimate}
		Let $d \geq 3$, $c \ne 0$ be such that $c<\lambda(d)$. Let $R>1$ and $\varphi_R$ be as in $(\ref{define varphi_R})$. Let $u: J\times \R^d \rightarrow \C$ be a radial solution to the $L^2$-critical $(\ref{focusing NLS inverse square})$. Then for any $\eps>0$ and any $t \in J$, it holds that
		\begin{align}
			\begin{aligned}
			\frac{d^2}{dt^2} V_{\varphi_R}(u(t)) &\leq 16 E(u_0) - 4 \int_{|x|>R} \left( \chi_{1,R} - \frac{\eps}{d+2} \chi_{2,R}^{\frac{d}{2}} \right) |\nabla u(t)|^2 dx \\
				&\mathrel{\phantom{\leq 16E(u_0) }} + O\left( R^{-2} + \eps R^{-2} + \eps^{-\frac{2}{d-2}} R^{-2} \right),
			\end{aligned}
			\label{radial virial estimate}
		\end{align}
		where
		\begin{align}
			\chi_{1,R} = 2-\varphi''_R, \quad \chi_{2,R} = 2d-\Delta \varphi_R. \label{define chi_1 chi_2}
		\end{align}
	\end{lemma}
	We refer the reader to \cite[Lemma 5.6]{Dinh} for the proof of this result, which is based on the argument of \cite{OT}. 
	\begin{lemma}[blow-up criteria \cite{Dinh}] \label{lemma blow-up criteria}
		Let $d\geq 3$ and $c\ne 0$ be such that $c<\lambda(d)$. If $u_0 \in H^1$ is radial and satisfies
		\[
		E(u_0)<0,
		\]
		then the corresponding solution $u \in C([0,T), H^1)$ to the $L^2$-critical $(\ref{focusing NLS inverse square})$ blows up in finite time, i.e. $T<+\infty$. 
	\end{lemma}
	\begin{proof}
		The proof of this result is given in \cite[Theorem 1.3]{Dinh}. For reader's convenience, we recall some details. Applying Lemma $\ref{lemma radial virial estimate}$, we have
		\[
		\frac{d^2}{dt^2} V_{\varphi_R} (u(t)) \leq 16 E(u_0) - 4 \int_{|x|>R} \left( \chi_{1,R} - \frac{\eps}{d+2} \chi_{2,R}^{\frac{d}{2}} \right) |\nabla u(t)|^2 dx + O \left( R^{-2} + 
		\eps R^{-2} + \eps^{-\frac{2}{d-2}} R^{-2} \right).
		\]
		If we choose a suitable function $\varphi_R$ so that 
		\begin{align}
		\chi_{1,R} - \frac{\eps}{d+2} \chi_{2,R}^{\frac{d}{2}} \geq 0, \quad \forall r >R, \label{non-negative condition}
		\end{align}
		for a sufficiently small $\eps>0$, then by choosing $R>1$ sufficiently large depending on $\eps$, we obtain
		\[
		\frac{d^2}{dt^2} V_{\varphi_R}(u(t)) \leq 8 E(u_0) <0,
		\]
		for any $t\in [0,T)$. The standard convexity argument then implies $T<+\infty$ or the solution blows up in finite time. Let us now choose $\varphi_R$ so that $(\ref{non-negative condition})$ is satisfied. To do this, we introduce the smooth function $\vartheta:[0,\infty) \rightarrow [0,\infty)$ satisfying
		\[
		\vartheta(r):= \left\{
		\begin{array}{cl}
		2r &\text{if } 0 \leq r\leq 1, \\
		2[r-(r-1)^3] &\text{if } 1<r \leq 1+1/\sqrt{3}, \\
		\vartheta' <0 &\text{if } 1+1/\sqrt{3} <r <2, \\
		0 &\text{if } r\geq 2,
		\end{array}
		\right.
		\]
		and define
		\[
		\theta(r):= \int_0^r \vartheta(s) ds. 
		\]
		It is not hard to check that $\theta$ satisfies $(\ref{choice of theta})$ and the function $\varphi_R$ defined as in $(\ref{define varphi_R})$ satisfies $(\ref{non-negative condition})$. We refer the reader to \cite{Dinh} for more details.
	\end{proof}
	We also need the so called Brezis-Lieb's lemma (see \cite{BL}). 
	\begin{lemma} [Brezis-Lieb's lemma \cite{BL}] \label{lemma brezis lieb lemma}
		Let $0<p<\infty$. Suppose that $f_n \rightarrow f$ almost everywhere and $(f_n)_{n\geq 1}$ is a bounded sequence in $L^p$, then
		\[
		\lim_{n \rightarrow \infty} \left( \|f_n\|^p_{L^p} - \|f_n-f\|^p_{L^p} \right) = \|f\|^p_{L^p}.
		\]
	\end{lemma}
	We are now able to prove the instability result given in Theorem $\ref{theorem instability II}$.

	\noindent {\it Proof of Theorem $\ref{theorem instability II}$.}
	Let $(\mu_n)_{n\geq 1}$ and $(\lambda_n)_{n\geq 1}$ be sequence of positive real numbers satisfying $\mu_n>1$,  $\lim_{n\rightarrow \infty} \mu_n =1$ and $\lim_{n\rightarrow \infty} \lambda_n =1$. Let $\omega>0$ and $Q$ be a radial $H^1$-solution to the elliptic equation $(\ref{mass-critical elliptic equation})$. Denote
	\begin{align}
	u_{0,n}(x) := \mu_n \lambda_n^{\frac{d}{2}} Q(\lambda_n x). \label{define initial data}
	\end{align}
	It is easy to see that
	\[
	\|u_{0,n}\|_{L^2}  = \mu_n \|Q\|_{L^2}, \quad \|\nabla u_{0,n}\|_{L^2}= \mu_n \lambda_n \|\nabla Q\|_{L^2}, \quad \|u_{0,n}\|_{\dot{H}^1_c} = \mu_n \lambda_n \|Q\|_{\dot{H}^1_c},
	\]
	for all $n\geq 1$. Moreover,
	\begin{align*}
	\lim_{n\rightarrow \infty} \|u_{0,n}\|_{L^2} = \lim_{n\rightarrow \infty} \mu_n \|Q\|_{L^2} = \|Q\|_{L^2}.
	\end{align*}
	By Lemma $\ref{lemma brezis lieb lemma}$, we see that
	\[
	u_{0,n} \rightarrow Q \text{ strongly in } L^2.
	\]
	Similarly,
	\[
	\lim_{n\rightarrow \infty} \|\nabla u_{0,n}\|_{L^2} = \lim_{n\rightarrow \infty} \mu_n \lambda_n \|\nabla Q\|_{L^2} = \|\nabla Q\|_{L^2}.
	\]
	Lemma $\ref{lemma brezis lieb lemma}$ again implies 
	\[
	u_{0,n} \rightarrow Q \text{ strongly in } \dot{H}^1,
	\]
	as $n\rightarrow \infty$. Therefore, $u_{0,n} \rightarrow Q$ strongly in $H^1$ as $n\rightarrow \infty$. It remains to show that $u_n$ blows up in finite time for $n$ sufficiently large. By Lemma $\ref{lemma blow-up criteria}$, it suffices to show that 
	\begin{align}
	E(u_{0,n}) <0, \label{negative energy}
	\end{align}
	for all $n \geq 1$. To see $(\ref{negative energy})$, we use $(\ref{pohozaev identity omega})$ to have
	\begin{align*}
	E(u_{0,n}) &= \frac{1}{2} \|u_{0,n}\|^2_{\dot{H}^1_c} - \frac{d}{2d+4} \|u_{0,n}\|^{\frac{4}{d}+2}_{L^{\frac{4}{d}+2}} \\
	&= \frac{1}{2} \mu_n^2 \lambda^2_n \|Q\|^2_{\dot{H}^1_c} - \frac{d}{2d+4} \mu_n^{\frac{4}{d}+2} \lambda_n^2 \|Q\|^{\frac{4}{d}+2}_{L^{\frac{4}{d}+2}} \\
	&= \frac{1}{2} \left(1-\mu_n^{\frac{4}{d}}\right) \mu_n^2 \lambda_n^2 \|Q\|^2_{\dot{H}^1_c}.
	\end{align*}    
	By the choice of $\mu_n$ and $\lambda_n$, we conclude that $(\ref{negative energy})$ holds for any $n\geq 1$. The proof is complete.
	\defendproof
	
	\appendix
	\section*{Appendix}
	In this short appendix, we will justify $(\ref{elliptic equation})$ and the radial symmetry of the limit in the compact embedding $(\ref{compact embedding})$.
	
	Let us first justify $(\ref{elliptic equation})$. Let $v \in S_M$, that is, $v \in H^1$, $\|v\|^2_{L^2} = M$ and $E(v)=d_M$. For $s \in \R$ and $\varphi$ a test function, we set
	\[
	v_s:= v + s \varphi, \quad \lambda_s:= \frac{\|v\|_{L^2}}{\|v_s\|_{L^2}} \quad \text{and} \quad w_s:= \lambda_s v_s. 
	\]
	We see that $\|w_s\|^2_{L^2} = \|v\|^2_{L^2} = M$, hence $E(w_s) \geq d_M$ for all $s \in \R$. We will prove that $F: s \mapsto E(w_s)$ is differentiable at $s=0$. Since the minimum of this function is attained at $s=0$, we get $F'(0) =0$. On the other hand, we have from a direct computation that
	\begin{align*}
	E(w_s) = \frac{\lambda_s^2 }{2} \|\nabla v_s\|^2_{L^2} -\frac{c \lambda_s^2 }{2} \||x|^{-1} v_s\|^2_{L^2} -\frac{\lambda_s^{\alpha+2}}{\alpha+2} \|v_s\|^{\alpha+2}_{L^{\alpha+2}},
	\end{align*}
	and
	\begin{align*}
	\frac{d}{ds}\Big|_{s=0} \|\nabla v_s\|^2_{L^2} &= 2 \re{ \scal{-\Delta v, \varphi} }, \quad \frac{d}{ds}\Big|_{s=0} \||x|^{-1} v_s\|^2_{L^2} = 2 \re{\scal{|x|^{-2} v, \varphi}}, \\
	\frac{d}{ds}\Big|_{s=0} \|v_s\|^{\alpha+2}_{L^{\alpha+2}} &= (\alpha+2) \re{\scal{|v|^\alpha v, \varphi}},
	\end{align*}
	as well as
	\[
	\frac{d}{ds}\Big|_{s=0} \lambda_s^2 = -\frac{2}{\|v\|^2_{L^2}} \re{\scal{v,\varphi}}, \quad \frac{d}{ds}\Big|_{s=0}  \lambda_s^{\alpha+2} = -\frac{\alpha+2}{\|v\|^2_{L^2}} \scal{v, \varphi}.
	\]
	Inserting the above identities to $F'(0)= \left.\frac{d}{ds} \right|_{s=0} E(w_s)=0$, we obtain
	\[
	\re{\scal{-\Delta v - c|x|^{-2}v - \frac{E(v)}{\|v\|^2_{L^2}} v - |v|^\alpha v, \varphi}}=0.
	\]
	Testing the above equality with $i\varphi$ instead of $\varphi$ and using the fact $\re{(iz)}= -\im{(z)}$, we get
	\[
	\im{\scal{-\Delta v - c|x|^{-2}v - \frac{E(v)}{\|v\|^2_{L^2}} v - |v|^\alpha v, \varphi}}=0.
	\]
	Therefore, $v$ solves
	\[
	-\Delta v - c|x|^{-2}v - \frac{d_M}{M} v - |v|^\alpha v =0.
	\]
	This solves $(\ref{elliptic equation})$ with $\omega = -\frac{d_M}{M}>0$.
	
	We next justify the radial symmetry of the limit in the compact embedding $(\ref{compact embedding})$. It is well-known that if $(u_n)_{n\geq 1}$ is a bounded sequence of $H^1_{\text{rad}}$ functions, then there exist a subsequence still denoted by $(u_n)_{n\geq 1}$ and $u \in H^1$ such that
	\[
	u_n \rightharpoonup u \text{ weakly in } H^1 \text{ and } u_n \rightarrow u \text{ strongly in } L^p, 
	\]
	for any $2<p<\frac{2d}{d-2}$. We will show that $u$ is radial. Indeed, since $u_n$ is radial, there exists $f_n: \R^+ \rightarrow \C$ such that $u_n(x) = f_n(|x|)$. The sequence $(u_n)_{n\geq 1}$ is strongly convergent in $L^p$, so it is a Cauchy sequence in $L^p$. This implies that $(f_n)_{n\geq 1}$ is a Cauchy sequence in $L^p(\R^+, r^{d-1} dr)$, hence strongly convergent since the latter space is complete. That is, there exists $f \in L^p(\R^+, r^{d-1} dr)$ such that
	\[
	f_n \rightarrow f \text{ strongly in } L^p(\R^+, r^{d-1} dr).
	\]
	We infer that there exists a subsequence $(f_{n_k})_{k\geq 1}$ of $(f_n)_{n\geq 1}$ and $I \subset \R^+$ with $\int_{I} r^{d-1} dr =0$ such that for all $r \in \R^+ \backslash I$,
	\[
	f_{n_k}(r) \rightarrow f(r),
	\]
	as $k \rightarrow \infty$. On the other hand, since $u_{n_k} \rightarrow u$ strongly in $L^p$, there exists a subsequence $(u_{n_{k_l}})_{l\geq 1}$ of $(u_{n_k})_{k\geq 1}$ and $\Omega \subset \R^d$ with $\int_{\Omega} dx =0$ such that for all $x \in \R^d \backslash \Omega$,
	\[
	u_{n_{k_l}} (x) \rightarrow f(|x|),
	\]
	as $l \rightarrow \infty$. Set $X:= \Omega \cup \{x \in \R^d \ : \ |x| \in I\}$. One easily checks that $X$ has Lebesgue measure zero. Let $x \in \R^d \backslash X$. On one hand, $u_{n_{k_l}}(x) \rightarrow u(x)$ as $l \rightarrow \infty$ (since $x \notin \Omega$), and on the other hand, $u_{n_{k_l}}(x)= f_{n_{k_l}}(|x|) \rightarrow f(|x|)$ as $l \rightarrow \infty$ (since $|x| \notin I$). One thus get $u(x) = f(|x|)$ or $u$ is radial.
	
    \section*{Acknowledgments}
    This work was supported by the National Natural Science Foundation of China grant No. 11501395. The first author would like to thank his thesis advisor, Pr. Sahbi Keraani, for his suggestions and encouragement. The second author would like to express his deep gratitude to his wife-Uyen Cong for her encouragement and support. The authors would like to thank the reviewers for their helpful comments and suggestions.

\end{document}